\documentclass[12pt,a4paper,leqno]{article}
\usepackage{exscale}
\usepackage{a4}
\usepackage{amssymb}
\usepackage{latexsym}

\newtheorem{theorem}{\bf Theorem}
\newtheorem{lemma}[theorem]{\bf Lemma}

\newtheorem{corollary}[theorem]{\bf Corollary}
\newtheorem{example}[theorem]{\bf Example}
\newtheorem{remark}[theorem]{\bf Remark}
\newenvironment{proof}{{\fontshape{sc} Proof.}}{\hfill $\Box$}

\newcommand{\nsubsection}{\setcounter{equation}{0}\subsection}

\begin{document}

\title {Fast approximation of  solutions of SDE's
with oblique reflection on an orthant
}
\author {Krzysztof T. Czarkowski (Toru\'n)\thanks{e-mai: kczark@mat.umk.pl}}
\maketitle

\begin{abstract}
We consider the discrete "fast" penalization scheme for SDE's driven
by general semimartingale on orthant $\mathbb{R}_{+}^{d}$ with
oblique reflection.
\end{abstract}
{\em Mathematics Subject Classification}: Primary 60 H 20; Secondary 60 H 99.

\nsubsection{Introduction} \label{introd}
 Suppose we have a
$d$\--dimensional semimartingale $Z=(Z^1,\ldots ,Z^d)^T$, a
Lipschitz continuous function
$\sigma:\mathbb{R}^{d}\longrightarrow\mathbb{R}^{d}\otimes\mathbb{R}^{d}$,
and a nonnegative $d\times d$ matrix $Q$ with zeros on the diagonal
and spectral radius $\rho(Q)$ strictly less than 1. Consider a
$d$\--dimension stochastic differential equation (SDE) on orthant
$\mathbb{R}_{+}^{d}$ with oblique reflection:
\begin{equation}
 \label{eq1}
 X_t=X_0 +\int_0^t\sigma(X_{s-})dZ_{s}+(1-Q^{T})K_{t},\;\; t\in\mathbb{R}_{+} .
\end{equation}
Equation of this type (\ref{eq1}) was introduced by Harrison and
Reiman \cite{hr}. Later it was discussed by Dupis and Ishi
\cite{di}. Czarkowski and S\l omi\'{n}ski \cite{cs} in their paper
introduced a numerical scheme for approximation of solution of SDE
(\ref{eq1}). In this paper we will define a new numerical scheme,
see: Section \ref{scheme} (\ref{eq3}). For this scheme (or it's
equivalent form (\ref{eq5})), in Section \ref{fasfd},  we will
proove, that:
\[
E\sup_{s\leq t} |X^{n}_{s}-X_{s}|^{2p}={\cal O}((\frac{\ln n}{n})^p)
\]

Appendix A includes a description of  some properties of $\Pi_{Q}$
projection on the orthant $\mathbb{R}_{+}^{d}$.

Throughout the paper we assume that
$\rho^{n}_{t}=\max\{i/n;i\in{\mathbb N}\cup\{0\},i/n\leq t\}$ and
$Z^{(n)}_{t}$ is a discretization of $Z$, ie.
$Z^{(n)}_{t}=Z_{\rho^{n}_{t}}$; $\mathop{\longrightarrow}_{\cal P}$
denotes convergence in probability;
$\mathbb{D}(\mathbb{R}_{+}\,,\,\mathbb{R}^{d})$ denotes space of
{\em "cadlag"} function
$y:\mathbb{R}_{+}\longrightarrow\mathbb{R}^{d}$; $\Delta
y_{t}=y_{t}-y_{t-}$ and $\omega_{\frac{1}{n}}(y,[0,t])$ denotes
modulus of continuity of $y$ on $[0,t]$.

Let us define function: $[z]^{+}=\max\{z,0\}$ for $z\in\mathbb{R}$
and by analogy function for
$z=(z^{1},\ldots,z^{d})^{T}\in\mathbb{R}^{d}$: $[z]^{+}= \left(
\left[z^{1}\right]^{+},\ldots,\left[z^{d}\right]^{+} \right)^{T}.$
We will use  norm $\|Q\|=\max_{1\leq i\leq d}\sum_{j=1}^{d} q_{ij}.$

 \nsubsection{The Skorokhod problem on an orthant}
 \label{skor}
 Let $Q$ be a nonnegative  matrix with
zeros on the diagonal and spectral radius $\rho(Q)<1$ and let
$y\in\mathbb{D}(\mathbb{R}_{+}\,,\,\mathbb{R}^{d})$ with
$y_0\in\mathbb{R}_{+}^{d}$. Following Harrison and Reiman \cite{hr}
a pair $(x,k)\in\mathbb{D}(\mathbb{R}_{+}\,,\,{\mathbb R}^{2d})$ is
called a solution to the Skorokhod problem
\begin{equation}
\label{eq2}
 x_{t}=y_{t}+(I-Q^{T})k_{t},\quad t\in\mathbb{R}_{+},
\end{equation}
on $\mathbb{R}_{+}^{d}$ associated with $y$, if (\ref{eq2}) is
satisfied and
\begin{eqnarray*}
& & x_{t}\in \mathbb{R}_{+}^{d},\quad t\in\mathbb{R}_{+},\\
& & k^{j}\, \mbox{\rm is nondecreasing},\, k^{j}_0=0\,\,\mbox{\rm
and} \,\int_{0}^{t} x^{j}_{s}\,dk^{j}_{s}=0\,\, \mbox{\rm for}\,\,
j=1,\ldots,d,\,\,t\in\mathbb{R}_{+}.
\end{eqnarray*}

\begin{remark}
\label{rm1}
\begin{enumerate}
\item For every $y\in\mathbb{D}(\mathbb{R}_{+}\,,\,\mathbb{R}^{d})$ with $y_{0}\in\mathbb{R}_{+}^{d}$ exist a unique
solution $(x_{t},k_{t})$ of the Skorokhod problem.
\item If additionally $\|Q\|<1$ then $k_{t}$  satisfy equation
\begin{equation}
k_{t}=F(k)_{t},
\end{equation}
where \[F(u)_{t}=\sup_{s\leq t}[Q^{T} u_{s} -y_{s}]^{+}.\]
\end{enumerate}
\end{remark}

In this paper, like in \cite{hr} and \cite{cs}, we make a technical
assumption that:
\begin{equation} \label{ta}\|Q\|<1.\end{equation}

\nsubsection{Fast approximation scheme} \label{scheme}
 Let $(x,k)$
be a solution to the Skorokhod problem for
$y\in\mathbb{D}(\mathbb{R}_{+}\,,\,\mathbb{R}^{d})$, with
$y_0\in\mathbb{R}_{+}^{d}$.

For every $n\in{\mathbb N}$ we define the approximations
$(x^{n},k^{n})$ of $(x,k)$:
\begin{equation}
\label{eq3}
\left\{
\begin{array}{rcl}
k^{n}_{0}&=&0,\;\;\; x^{n}_{0}=y_{0},\\
k^{n}_{(i+1)/n}&=&[Q^{T}k^{n}_{i/n}-y_{(i+1)/n}]^{+}\vee k^{n}_{i/n},\\
x^{n}_{(i+1)/n}&=&y_{(i+1)/n}+(I-Q^{T})k^{n}_{(i+1)/n},\\
k^{n}_{t}&=&k^{n}_{i/n},\;\; x^{n}_{t}=x^{n}_{i/n}, \mbox{ for
}t\in[\frac{i}{n},\frac{i+1}{n}).
\end{array}
\right.
\end{equation}

\begin{remark}
We can write another, but equivalent form of $k^{n}$,
$x^{n}$.\newline Note that for every $n\in{\mathbb N}$,
$i\in{\mathbb N}\cup\{0\}$:
\begin{eqnarray}
\label{eq4}
 k^{n}_{(i+1)/n}&=&[(Q^{T}-I)k^{n}_{i}-y_{(i+1)/n}]^{+}+k^{n}_{i/n}\\
 \nonumber
 &=&[-(x^{n}_{i/n}+\Delta y_{(i+1)/n})]^++k^{n}_{i/n},\\
\label{eq5} x^{n}_{(i+1)/n}&=&x^{n}_{i/n}+\Delta
y_{(i+1)/n}+(I-Q^{T})[-x^{n}_{i/n}-\Delta y_{(i+1)/n}]^{+},
\end{eqnarray}
where $\Delta y_{(i+1)/n}=y_{(i+1)/n}-y_{i/n}$.
\end{remark}

Formulas (\ref{eq3}) and (\ref{eq5}) are equivalent, but (\ref{eq5})
is simpler to calculate.
\begin{remark}
\label{rm4} We can see that $k^{n}_{t}$ satisfy equation:
\begin{equation}
\label{eq6}k^{n}_{t}=F^{n}(k^{n,(n-)})_{t}, \end{equation}
 where
$F^n(u)_{t}=\sup_{s\leq t}[Q^{T} u_{s}-y^{(n)}_{s}]^{+}$ and
$u_{t}^{(n-)}=u_{(i-1)/n},\;\;\;t\in[{i/n},{(i+1)/n}).$
\end{remark}

 The next two theorems describe some properties of
scheme (\ref{eq3}).
\begin{theorem}
\label{th3} There exist a constant ${\cal C}>0$ depending only on
$Q$ such that for every
$y\in\mathbb{D}(\mathbb{R}_{+}\,,\,\mathbb{R}^{d})$,
$y_{0}\in\mathbb{R}_{+}^{d}, t\in \mathbb{R}_{+}$:
\begin{equation}
\label{lab34} \sup_{s\leq t}|x^{n}_s-x_s|+\sup_{s\leq t}|k^{n}_{s}
-k_{s} |\leq {\cal C} \omega_{\frac{1}{n}}(y,[0,t]).
\end{equation}
\end{theorem}
\begin{proof}
Since $\sup_{s\leq t}|x^{n}_s-x_s|< ||Q|| \sup_{s\leq t}|k^{n}_{s}
-k_{s} |+\omega_{\frac{1}{n}}(y,[0,t])$ we estimate only first term
i.e. $\sup_{s\leq t}|k^{n}_{s} -k_{s} |$.

 We assumed that (\ref{ta}) is satisfied, that means that
$\|Q\|<1$.

Firstly we proof (\ref{lab34}). From Remarks \ref{rm1} and
\ref{rm4}:
\begin{eqnarray*}
&&\sup_{s\leq t} |k^{n}_s-k_s|=\sup_{s\leq t}
|F^n(k^{n,(n-)})_{s}-F(k)_s|\\ &\leq& \sup_{s\leq t}
|F^n(k^{n,(n-)})_s-F^n(k^{n})_s|
+\sup_{s\leq t} |F^n(k^{n})_s-F(k^{n})_s|\\ &&+\sup_{s\leq t} |F(k^{n})_s -F(k)_s|\\
&=&I_t^1+I_t^2+I_t^3
\end{eqnarray*}
Now we estimate every part separately:
\begin{eqnarray*}
I_t^1&=& \sup_{s\leq t} |F^n(k^{n,(n-)})_s -F^n(k^{n})_s|\leq
||Q||\max_{\frac{i}{n}\leq t} |k^{n}_{(i-1)/n}-k^{n}_{i/n}|\\
&\leq& ||Q||^{2}\max_{\frac{i}{n}\leq
t}|k^{n}_{(i-2)/n}-k^{n}_{(i-1)/n}| +||Q||\max_{\frac{i}{n}\leq
t}|y_{(i-1)/n}-y_{i/n}|\\ &\leq&
\frac{||Q||}{1-||Q||}\omega_\frac{1}{n}(y,[0,t]),\\
I_t^2&=&\sup_{s\leq t} |F^n(k^{n})_s -F(k^{n})_s|\\ & \leq&
\sup_{s\leq t} |[Q^{T}k^{n}_{s}-y_{s}^{(n)} ]^+ -[Q^Tk^{n}_{s}-y_{s}
]^+|\\ & \leq &\sup_{s\leq t}
|y^{(n)}_{s}-y_{s}| = \omega_\frac{1}{n}(y,[0,t]),\\
I_t^3&=&\sup_{s\leq t} |F(k^{n})_s -F(k)_s| \leq ||Q||\sup_{s\leq t}
|k^{n}_{s}-k_{s}|
\end{eqnarray*}
And now we have:
\[
\sup_{s\leq t} |k^{n}_{s}-k_{s}|
\leq\frac{||Q||}{1-||Q||}\omega_{\frac{1}{n}}(y,[0,t])+\omega_{\frac{1}{n}}(y,[0
,t])+||Q||\sup_{s\leq t}|k^{n}_{s}-k_{s}|
\]
\end{proof}

\begin{corollary}
For every  $y\in\mathbb{C}(\mathbb{R}_{+}\,,\,\mathbb{R}^{d})$,
$y_{0}\in\mathbb{R}_{+}^{d}$ we have
\[\sup_{s\leq t}|x^{n}_{s}-x_{s}|+\sup_{s\leq t}|k^{n}_{s}-k_{s}|\rightarrow 0.\]
\end{corollary}

\begin{theorem}
\label{th4} There exist a constant ${\cal C}>0$ depending only on
$Q$ such that for every
$y^{1},y^{2}\in\mathbb{D}(\mathbb{R}_{+}\,,\,\mathbb{R}^{d})$,
$y^{1}_{0},y^{2}_{0}\in\mathbb{R}_{+}^{d}$ :
\[
 \sup_{s\leq t}|k_{s}^{1,n} -k_{s}^{2,n} |+\sup_{s\leq t}|x_{s}^{1,n}-x_{s}^{2,n}|\leq {\cal C} \sup_{s\leq
t}|y^{1}_{s}-y^{2}_{s}|.
\]
\end{theorem}
\begin{proof}
Like in previous theorem we only need to proof first term of
theorem. The second we obtain from (\ref{eq3}).
\begin{eqnarray*} \sup_{s\leq t}|k^{1,n}_{s} -k^{2,n}_{s} |&=&
\sup_{s\leq t} |F^n(k^{1,n,(n-)})_{s}-F^n(k^{2,n,(n-)})_s|\\ &=&
\sup_{s\leq t} |[Q^{T} k^{1,n,(n-)}_{s}-y^{1,(n)}_{s}]^{+}-[Q^{T}
k^{2,n,(n-)}_{s}-y^{2,(n)}_{s}]^{+} |\\ &\leq& \|Q^{T}\|
\max_{\frac{i}{n}\leq t}|k^{1,n}_{(i-1)/n}-k^{2,n}_{(i-1)/n}|
+\max_{\frac{i}{n}\leq t}|y^{1}_{i/n}-y^{2}_{i/n}| \\ &\leq&
\|Q^{T}\| \sup_{s\leq t}|k^{1,n}_{s}-k^{2,n}_{s}| +\sup_{s\leq
t}|y^{1}_{s}-y^{2}_{s}|
\end{eqnarray*}
and
\[
\sup_{s\leq t}|k^{1,n}_{s} -k^{2,n}_{s} |\leq
\frac{1}{1-\|Q^{T}\|}\sup_{s\leq t}|y^{1}_{s}-y^{2}_{s}| .\]
\end{proof}

Easy corollary:
\begin{corollary}
There exists a constant ${\cal C}>0$ such that for every
$y\in\mathbb{D}(\mathbb{R}_{+}\,,\,\mathbb{R}^{d})$,
$y_{0}\in\mathbb{R}_{+}^{d}$:
\[
 k^{n}_{t}\leq {\cal C} \sup_{s\leq t} |y_{s}|<
+\infty.
\]
\end{corollary}

From previous Theorems we can obtain convergent for continous
functions.

In the next lemma and theorem we formalize this observation.

\begin{lemma}
\label{lem2} Let
$y\in\mathbb{D}(\mathbb{R}_{+}\,,\,\mathbb{R}^{d})$,
$y_{0}\in\mathbb{R}_{+}^{d}$ has the form
\begin{equation}
\label{lm9} y_{t}=\sum_{i=0}^{+\infty} y_{t_{i}}{\bf
1}_{[t_{i},t_{i+1})}(t),
\end{equation}
 where $0=t_{0}<t_{1}<\ldots$, then:
\begin{equation}
\label{lm9t} x^{n}_{t}\mathop{\longrightarrow}_{n\rightarrow\infty}
x_{t} \end{equation}
 for $t\neq t_{i}$, $i\in{\mathbb N}$, where
$(x_{t},k_{t})$ is a solution of the Skorokhod Problem for $y_{t}$.
\end{lemma}
\begin{proof}
We proof lemma "by induction". It's well known's that, if
 $y$ is of the form (\ref{lm9}) then:
\[
x_{t}= \left\{
 \begin{array}{ll}
 y_{0}; & t\in[0,t_{1})\\
 \Pi_{Q}(x_{t_{i-1}}+\Delta y_{t_{i}});& t\in[t_{i},t_{i+1}),\;
 i\in{\mathbb N}.
 \end{array} \right.
\]

\begin{description}
\item{1.} So for $t\in[0,t_{1})$ thesis is satisfy by definition.
\item{2.} By scheme (\ref{eq5}):
\[
x^{n}_{(i+1)/n}= \left\{
\begin{array}{ll}
x^{n}_{i/n}+\Delta y_{(i+1)/n}+(I-Q^{T})&[-x^{n}_{i/n}-\Delta
y_{(i+1)/n}]^{+}, \\
 & \mbox{ for } i \mbox{ such that } \frac{i}{n}\leq t_{i}<\frac{i+1}{n}\\
 & \\
x^{n}_{i/n}+(I-Q^{T})[-x^{n}_{i/n}]^{+},& \mbox{ for } i \mbox{ such
that } t_{i}<\frac{i}{n}< t_{i+1}
\end{array}
 \right.
\]

Between jumps in points $t_{i}$ sequence $x_{i}^{n}=x_{i/n}^{n}$ has
form like $z_{i}$ (see from Appendix A (\ref{lab20})) starting from
$z_{0}=(x_{t_{i}}+\Delta y_{t_{i}})$.

Then for $n\rightarrow +\infty$ from  Corollary \ref{cor26}:
$\lim_{n\rightarrow+\infty}x^{n}_{t}= \Pi_{Q}(x_{t_{i}}+\Delta
y_{t_{i}})\;\;\mbox{ for } t\in(t_{i},t_{i+1}).$
\end{description}
\end{proof}

In the next example we show that (\ref{lm9t}) can't be streightend
to the covergent in the Skorokhod topology $J_{1}$.

\begin{example}
\label{ex1}
{\em Let $d=2$, \[Q=
 \left\{
  \begin{array}{cc}
  0          & \frac{1}{2}\\
  \frac{1}{2}& 0
  \end{array}
 \right\}\mbox{ and }y_{t}=
 \left\{
  \begin{array}{cc}
   ( 0, 0)^T;&\;\; t<1\\
   (-1,-1)^T;&\;\; t\geq 1
  \end{array}
 \right..
\]
The solution of Skorokhod problem are functions:
\[x_{t}=(0,0)^T\;\;\; t\in\mathbb{R}_{+}\mbox{ and } k_{t}=\left\{
  \begin{array}{cc}
   (0,0)^T;&\;\; t<1\\
   (2,2)^T;&\;\; t\geq 1 .
  \end{array}
 \right.\]
 }

{\em  Now, we use scheme (\ref{eq3}) for this function and try to
find the limit of $(x^{n},k^{n})$ when $n$ tends to infinity. }

{\em
 When we use scheme (\ref{eq3}) we obtain:
\[k^{n}_{t}=\left\{
 \begin{array}{ll}
 (              0,              0)^T; &\;\; t<1\\
 (1,1)^T; &\;\; t\in[1,1+\frac{1}{n})\\
 (2-\frac{1}{2^i},2-\frac{1}{2^i})^T; &\;\; t\in[1+\frac{i}{n},1+\frac{i+1}{n}),\;\ i\in{\mathbb N}
 \end{array}
\right.\] and
\[x^{n}_{t}=\left\{
 \begin{array}{ll}
  (                 0,                 0)^T; &\;\; t<1\\
  (-\frac{1}{2},-\frac{1}{2})^T; &\;\; t\in[1,1+\frac{1}{n})\\
  (-\frac{1}{2^{i+1}},-\frac{1}{2^{i+1}})^T; &\;\; t\in[1+\frac{i}{n},1+\frac{i+1}{n}),\;\ i\in{\mathbb N}
 \end{array}
\right..\]

Since $\sup_{t\leq 2}|x^{n}_{t}|=\frac{1}{2}$ the solution
$x^{n}\not\longrightarrow x$ in $J_{1}$. }
\end{example}

\begin{corollary}
\label{rem1} If $y$ satisfy assumption of Lemma \ref{lem2}, then:
\begin{equation}
(x^{n},k^{n})\longrightarrow(x,k)\mbox{ in }
(\mathbb{D}(\mathbb{R}_{+}\,,\,{\mathbb R}^{2d}), S)
\end{equation}
\end{corollary}
\begin{proof}
From Jakubowski \cite{ja} Lemma 2.14.
\end{proof}

\begin{theorem}
If $y\in \mathbb{D}(\mathbb{R}_{+}\,,\,\mathbb{R}^{d})$ and
$y_{0}\in\mathbb{R}_{+}^{d}$, then
\begin{equation}
\label{eq306} (x^{n},k^{n})\longrightarrow(x,k)\mbox{ in }
(\mathbb{D}(\mathbb{R}_{+}\,,\,{\mathbb R}^{2d}), S)
\end{equation}
\end{theorem}
\begin{proof}
For all $y\in\mathbb{D}(\mathbb{R}_{+}\,,\,\mathbb{R}^{d})$ and all
$\epsilon>0$ exist
$y^{\epsilon}\in\mathbb{D}(\mathbb{R}_{+}\,,\,\mathbb{R}^{d})$
satisfying assumption of Lemma \ref{lem2} such that $\sup_{s\leq
t}|y^{\epsilon}_{s}-y_{s}|\leq\epsilon$.

Let pair $(x^{\epsilon},k^{\epsilon})$ be a solution of the Skorokhod problem
for $y^{\epsilon}$. Then from Lemma \ref{lem2}:
\[
(x^{\epsilon,n},k^{\epsilon,n})\longrightarrow(x^{\epsilon},k^{\epsilon}))\mbox{
in } (\mathbb{D}(\mathbb{R}_{+}\,,\,{\mathbb R}^{2d}), S)
\]
to show thesis we need that:
\[
\lim_{\epsilon\longrightarrow0}\sup_{n}\sup_{s\leq
t}|k^{\epsilon,n}_{s}-k^{n}_{s}|=0
\]

From Theorem \ref{th4} we have:
\[
\sup_{s\leq t}|k^{\epsilon,n}_{s}-k^{n}_{s}|\leq{\cal C} \sup_{s\leq
t}|y^{\epsilon (n)}_{s}-y^{(n)}_{s}|\leq{\cal C}\epsilon
\]
\end{proof}
\begin{remark}
If $y\in \mathbb{D}(\mathbb{R}_{+}\,,\,\mathbb{R}^{d})$ and
$y_{0}\in\mathbb{R}_{+}^{d}$, then $ x^{n}\longrightarrow x$ for
continuity point of $y$ and $\{x^{n}\}$ is relatisvely $S$-compact.
\end{remark}

\nsubsection{Fast approximation scheme for SDE} Let $Z$ be a
$(({\cal F}_{t}))$-adapted semimartingale. Let us recall that pair
$(X,K)$ of $(({\cal F}_{t}))$-adapted processes is called {\em
strong solution} of (\ref{eq1}) if $(X,K)$ is a solution to the
Skorokhod problem associated with the semimartingale:
\begin{equation}
\label{eq501}
Y_{t}=X_0+\int_0^t\sigma(X_{s-})dZ_{s},\;\;\;t\in\mathbb{R}_{+} .
\end{equation}

\begin{remark}
If $\sigma$ is Lipschitz continuous, then there exist a unique
strong solution to the SDE (\ref{eq1}).
\end{remark}

Using formulas (\ref{eq3}) we can define "fast" scheme for SDE:
\[
\left\{
\begin{array}{rcl}
X^{n}_0&=&X_0,\;\;\;
K^{n}_0=0,\\
K^{n}_{(i+1)/n}&=&[Q^{T}K^{n}_{i/n}-(X^{n}_{i/n}+\sigma(X^{n}_{i/n})(Z_{(i+1)/n}-Z_{i/n}) ]^{+} \vee K^{n}_{i/n},\\
x^{n}_{(i+1)/n}&=&X^{n}_{i/n}+\sigma(X^{n}_{i/n})(Z_{(i+1)/n}-Z_{i/n})+(1-Q^{T})K^{n}_{(i+1)/n},\\
(X^{n}_{t},K^{n}_{t})&=&(X^{n}_{i/n},K^{n}_{i/n})\;\;\;
t\in[\frac{i}{n},\frac{i+1}{n}).
\end{array}
\right.
\]

\begin{lemma}
\label{lem17} Assume there exist stoping times
$\{\tau_{i}\}\subset\mathbb{R}_{+}$ such that:
$0=\tau_{0}<\tau_{1}<\ldots $ and $\{Z_{i}\}\subset\mathbb{R}^{d}$.
If $Z$ is semimartingale such
\[ Z_{t}=Z_{i}\]
for $ t\in[\tau_{i},\tau_{i+1}), \; i\in{\mathbb N}\cup\{0\}$ then
\begin{equation}
\label{eq503l} X^{n}_{t}\longrightarrow X_{t}\;\;\; for \; t\neq
t_{i}.
\end{equation}
\end{lemma}
\begin{proof}
We can write formula for $X_{t}$:
\[
X_{t}=\left\{
\begin{array}{ll}
X_{0}:&t\in[0,\tau_{1})\\
\Pi_{Q}(X_{\tau_{i-1}}+\sigma(X_{\tau_{i-1}})\Delta
Z_{\tau_{i}}):&t\in[\tau_{i},\tau_{i+1}),\; i\in{\mathbb N}
\end{array}
\right.
\]
The rest of proof is the same like in Lemma \ref{lem2}. We only need
to change $\Delta y^{(n)}_{(i+1)/n}$ by
$\sigma(X_{\tau_{i-1}})\Delta Z_{\tau_{i}}$.
\end{proof}

\begin{theorem}
Assume that $\sigma$ is Lipschitz continuous, then
\begin{equation}
\label{eq503} (X^{n},K^{n})\mathop{\longrightarrow}_{\cal P}
(X,K)\;\;\; in \; (\mathbb{D}(\mathbb{R}_{+}\,,\,{\mathbb
R}^{2d}),S).
\end{equation}
\end{theorem}
\begin{proof}
\[ \forall_{\epsilon >0} \exists_{Z^{\epsilon}} \sup_{s\leq t}|Z_{s}-Z_{s}^{\epsilon}|\leq \epsilon \]
From Lemma \ref{lem17}
\[({X^{\epsilon}}^{n},{K^{\epsilon}}^{n})\longrightarrow ({X^{\epsilon}},{K^{\epsilon}})\;\;\; in (\mathbb{D}(\mathbb{R}_{+}\,,\,{\mathbb R}^{2d}),S) \]

To proof we need:
\[ \lim_{\epsilon\rightarrow 0}\limsup_{n}P(\sup_{s\leq t}|{X^{\epsilon}_{s}}^{n}-{X}_{s}^{n}|>\eta)= 0 \;\;\;\forall_{\eta>0}\]

\begin{eqnarray*}
\sup_{s\leq t}|{X^{\epsilon}_{s}}^{n}-{X}_{s}^{n}|&\leq&{\cal C}\sup_{s\leq t}|{Y^{\epsilon}_{s}}^{n}-{Y}_{s}^{n}|\\
&=&{\cal C}\sup_{s\leq t}|\int_{0}^{s}\sigma({X}_{u-}^{n})dZ_{u}^{(n)}-\int_{0}^{s}\sigma({X^{\epsilon}_{u}}^{n}_{-})dZ_{u}^{\epsilon,(n)}|\\
&\leq&{\cal C}\sup_{s\leq t}|\int_{0}^{s}\sigma({X}_{u-}^{n})dZ_{u}^{(n)}-\int_{0}^{s}\sigma({X^{\epsilon}_{u}}^{n}_{-})dZ_{u}^{(n)}|\\
&+&{\cal C}\sup_{s\leq t}|\int_{0}^{s}\sigma({X^{\epsilon}_{u}}^{n}_{-})dZ_{u}^{(n)}-\int_{0}^{s}\sigma({X^{\epsilon}_{u}}^{n}_{-})dZ_{u}^{\epsilon,(n)}|\\
&=&{\cal C}\sup_{s\leq
t}|\int_{0}^{s}(\sigma({X}_{u-}^{n})-\sigma({X^{\epsilon}_{u}}^{n}_{-}))dZ_{u}^{(n)}|+{\cal
C}\sup_{s\leq t} |H_{s}^{\epsilon,n}|
\end{eqnarray*}
\begin{eqnarray*}
H_{t}^{\epsilon,n}&=&\int_{0}^{t}\sigma({X^{\epsilon}_{s}}^{n}_{-})dZ_{s}^{(n)}-\int_{0}^{t}\sigma({X^{\epsilon}_{s}}^{n}_{-})dZ_{s}^{\epsilon,(n)}\\
&=&\int_{0}^{t}\sigma({X^{\epsilon}_{s}}^{n}_{-})d(Z_{s}^{(n)}-Z_{s}^{\epsilon,(n)})\\
&=&\sigma({X^{\epsilon}_{t}}^{n})(Z^{(n)}_{t}-Z^{\epsilon,(n)}_{t})-\int_{0}^{t}(Z^{(n)}_{s-}-Z^{\epsilon,(n)}_{s-})d\sigma({X^{\epsilon}_{s}}^{n})\\
&&-[\sigma({X^{\epsilon}_{t}}^{n}),(Z^{(n)}_{t}-Z^{\epsilon,(n)}_{t})]
\end{eqnarray*}
\[
[\sigma({X^{\epsilon}_{t}}^{n}),(Z^{(n)}_{t}-Z^{\epsilon,(n)}_{t})]\leq
([\sigma({X^{\epsilon}_{t}}^{n})])^{1/2}([(Z^{(n)}_{t}-Z^{\epsilon,(n)}_{t})])^{1/2}
\]

\[{X^{\epsilon}_{t}}^{n}=X_{0}+\int_{0}^{t}\sigma ( {X^{\epsilon_{s-}^n}})dZ^{\epsilon,n}_{s}
+ (1-Q^{T}){K^{\epsilon}_{t}}^{n} \]

\begin{eqnarray*}
\sup_{s\leq t}|{X^{\epsilon}_{t}}^{n}| &\leq& |X_{0}|+\sup_{s\leq t}
|\int_{0}^{t}\sigma({X^{\epsilon_{u-}}}^n})dZ^{\epsilon,n}_{u}| +
(1-Q^{T})\sup_{s\leq t}|{K^{\epsilon_{s}^{n}}| \\ &\leq&
|X_{0}|+2{\cal C} \sup_{s\leq
t}|\int_{0}^{t}\sigma({X^{\epsilon_{u-}^n}})dZ^{\epsilon,n}_{u}|
\end{eqnarray*}

From Gronwall lemma we obtain, that $\{\sup|{X^{\epsilon^{n}}}|\}$
is bounded for $\sigma$ satisfying Lipshitz condition. So, if
$\{\sigma(|{X^{\epsilon^{n}}})\}$ is bounded in probability, then
\[ \int_{0}^{t}\sigma({X^{\epsilon_{s-}}}^{n})dZ^{\epsilon,n}_{s}
\]
satisfies $UT$ condition.

Because $\{K^{\epsilon^{n}}\}$ is bounded in probability, this means
that it also satysfies $UT$. $\{{X^{\epsilon^{n}}}\}$ satisfies $UT$
as a sum of two proceses that satisfy $UT$. So,
$\sigma({X^{\epsilon^{n}}})$ satisfies $UT$ for $\sigma \in C^{2}$.
\end{proof}

\nsubsection{Fast approximation scheme for diffusion} \label{fasfd}
Consider SDE with reflection on $\mathbb{R}_{+}^{d}$ of the form
\begin{equation}
\label{eq601}
X_t=X_0+\int_0^tb(X_s)ds+\int_0^t\sigma(X_s)dW_s+(1-Q^{T})K_t,
\end{equation}
where $W$ is $d$\--dimension Wiener process,
 and $b:\mathbb{R}^{d}\longrightarrow\mathbb{R}^{d}$,
$\sigma:\mathbb{R}^{d}\longrightarrow\mathbb{R}^{d}\otimes\mathbb{R}^{d}$.

\begin{remark}
If $b$ and $\sigma$ are Lipschitz continuous then  there exists a
unique strong solution of the SDE (\ref{eq601}).
\end{remark}

Let us define:
\[
\left\{
\begin{array}{r@{=}l}
X^{n}_0&X_0,\;\;\; K^{n}_0=0,\\
K^{n}_{(i+1)/n}&[Q^{T}K^{n}_{i/n}-(X^{n}_{i/n}+b(X^{n}_{i/n})\frac{1}{n}+\sigma(X^{n}_{i/n})(W_{(i+1)/n}-W_{i/n})]^+\vee
K^{n}_{i/n},\\
x^{n}_{(i+1)/n}&X^{n}_{i/n}+b(X^{n}_{i/n})\frac{1}{n}+\sigma(X^{n}_{i/n})(W_{(i+1)/n}-W_{i/n})+(1-Q^{T})K^{n}_{(i+1)/n},\\
\left(X^{n}_{t}, K^{n}_{t}\right)&\left(X^{n}_{i/n},
K^{n}_{i/n}\right)\;\;\; t\in [\frac{i}{n},\frac{i+1}{n}).
\end{array}
\right.
\]

We can see that $X^{n}$ satisfies equation:
\begin{equation}
\label{eq603} X^{n}_{t}=X^{n}_0+\int_0^t
b(X^{n}_{s-})d\rho^{n}_{s}+\int_0^t\sigma(X^{n}_{s-})dW^{(n)}_{s}+(1-Q^{T})K^{n}_{t}
.
\end{equation}

\begin{theorem}
\begin{equation}
E\sup_{s\leq t} |X^{n}_{s}-X_{s}|^{2p}={\cal O}((\frac{\ln n}{n})^p)
\end{equation}
\end{theorem}
\begin{proof}

\begin{lemma}
\begin{equation}
\label{eq604} \sup_{n}E\sup_{s\leq t} |X^{n}_{s}|^{2p}<+\infty
\end{equation}
\end{lemma}
\begin{proof}
\begin{equation}
\sup_{s\leq t}|X^{n}_{s}-X^{n}_0|\leq {\cal C} \sup_{s\leq
t}|\int_0^s\sigma(X^{n}_{u-})dW^{(n)}_u+\int_0^sb(X^{n}_{u-})d\rho^{n}_u|
\end{equation}
From this we derive:
\[
\sup_{s\leq t}|X^{n}_{s}-X^{n}_0|^{2p}\leq 2{\cal C} \sup_{s\leq
t}|\int_0^s\sigma(X^{n}_{u-})dW^{(n)}_u|^{2p}+2{\cal C}\sup_{s\leq
t}|\int_0^sb(X^{n}_{u-})d\rho^{n}_u|^{2p}
\]
Now because $b$ and $\sigma$ are Lipschitz we have
\begin{eqnarray*}
&&E\sup_{s\leq t}|X^{n}_{s}-X^{n}_0|^{2p}\\
 &\leq& 2{\cal C} E(\int_0^t\sigma(X^{n}_{s-})dW^{(n)}_{s})^{2p}+2{\cal C} E(\int_0^t|b(X^{n}_{s}|ds)^{2p}\\
& \leq& 2{\cal C} E\int_0^t\sigma^{2p}(X^{n}_{s-})d\rho^{n}_s+2 {\cal C} E\int_0^tb^{2p}(X^{n}_{s^-})d\rho^{n}_{s}\\
& \leq& {\cal C} E\int_0^t((X^{n}_{s-})^{2p}+1)d\rho^{n}_s\\
& \leq& {\cal C}(1+\int_0^t E\sup_{u\leq s}
|X^{n}_u-X^{n}_0|^{2p}ds)
\end{eqnarray*}
From Gronwall lemma we have the thesis.
\end{proof}

\[X^{n}_{t}-X_{t}=\int_0^t(\sigma(X^{n}_{s-})-\sigma(X_{s-}))dW^{(n)}_{s}
            +\int_0^t(b(X^{n}_{s-})-b(X_{s-}))d\rho^{n}_{s} +(1-Q^{T})(K^{n}_{t}-K_{t})) \]
because $b$ and $\sigma$ are Lipschitz  then:
\[ E\sup_{s\leq t} |X^{n}_{s}-X_{s}|^{2p} \leq
   2{\cal C} (E\sup_{s\leq t}|K^{n}_s-K_s|^{2p}+\int_0^tE\sup_{u\leq s} |X^{n}_u-X_u|^{2p}ds \]
From Gronwall Lemma:
\[E\sup_{s\leq t}|X^{n}_s-X_s|^{2p}\leq 2{\cal C} E \sup_{s\leq t} |K^{n}_s-K_s|^{2p} \]
In the same way we can proof that
\[ E\sup_{s\leq t} |X^{n}_s|^2 \leq {\cal C} E \sup_{s\leq t}|K^{n}_s|^2\]
Because
\[K^{n}_t=\sup_{s\leq t} [Q^{T}K^{n,(n-)} - (X^{n}_0+\int_0^sb(X^{n}_{u-})d\rho^{n}_u+\int_0^s\sigma(X^{n}_{u-})dW^{(n)}_u)]^+\]
And
\[K_t=\sup_{s\leq t} [Q^{T}K_s-(X^{n}_0+\int_{0}^{s}b(X_{u-})d\rho_u+\int_0^s\sigma(X_{u-})dW_u)]^+\]
we have
\begin{eqnarray*}
&&K^{n}_{t}-K_t=\\
&& \sup_{s\leq t} [Q^{T}K_{s}^{n,(n-)} -(X^{n}_0+\int_0^sb(X^{n}_{u-})d\rho^{n}_u+\int_0^s\sigma(X^{n}_{u-})dW^{(n)}_u)]^+\\
&&-\sup_{s\leq t} [Q^{T}K^{n}_{s} - (X^{n}_0+\int_0^sb(X^{n}_{u-})d\rho^{n}_u+\int_0^s\sigma(X^{n}_{u-})dW^{(n)}_u)]^+\\
&&+\sup_{s\leq t} [Q^{T}K^{n}_{s} - (X^{n}_0+\int_0^sb(X^{n}_{u-})d\rho^{n}_u+\int_0^s\sigma(X^{n}_{u-})dW^{(n)}_u)]^+\\
&&-\sup_{s\leq t} [Q^{T}K^{n}_{s} - (X_0+\int_0^sb(X_{u-})d\rho_u+\int_0^s\sigma(X_{u-})dW_u)]^+\\
&&+\sup_{s\leq t} [Q^{T}K^{n}_{s} - (X_0+\int_0^sb(X_{u-})d\rho_u+\int_0^s\sigma(X_{u-})dW_u)]^+\\
&&-\sup_{s\leq t} [Q^{T}K_{s} - (X_0+\int_0^sb(X_{u-})d\rho_u+\int_0^s\sigma(X_{u-})dW_u)]^+\\
&&=I_{t}^{1}+I_{t}^{2}+I_{t}^{3}
\end{eqnarray*}

\begin{eqnarray*}
 I_{t}^{1}&\leq&\sup_{s\leq t} [K_{s}^{n,(n-)}-K^{n}_{s}|\\
    &\leq&\sup_{s\leq t}|\sigma(X^{n}_{s-})(W_s-W^{(n)}_s)+b(X^{n}_{s-})(s-\rho^{n}_s)|\\
 I_{t}^{2}&\leq&\sup_{u\leq s}|\int_0^s(b(X^{n}_{u-})-b(X_{u-}))d\rho^{n}_u+\int_0^s((\sigma(X^{n}_{u-})-(\sigma(X_{u-}))dW^{(n)}_u\\
 I_{t}^{3}&\leq&\sup_{s\leq t} |K^{n}_s-K_s|\\
 \end{eqnarray*}
and we have
\begin{eqnarray*}
&&E\sup_{s\leq t} |K_s -K^{n}_s|^{2p}\\
&\leq& {\cal C} (E\sup_{s\leq t} |W_s-W^{(n)}_s|^{2p}+ \int_0^sE\sup_{u\leq s} |K_u-K^{n}_u|^{2p}du)\\
&\leq& {\cal C} E(\omega_{\frac{1}{n}}(W, [0,t]))^{2p}\\
&=&{\cal O}((\frac{ \ln n}{n})^p)
\end{eqnarray*}
\end{proof}

\nsubsection*{Appendix A: $\Pi_{Q}$ projection}

Finding the projection $\pi$ on the domain $D$ is the standard
techniquie to obtain solution of the Skorokhod Problem. In \cite{cs}
we define projection on the orthant $\mathbb{R}_{+}^{d}$:
\begin{remark}
\label{remA7}
$\Pi_{Q}:\mathbb{R}^{d}\longrightarrow\mathbb{R}_{+}^{d}$ is defined
by formula:
\[ \Pi_{Q}(z)=z+(I-Q^{T})\bar{r}, \]
where $\bar{r}$ satisfy equation:
\[\bar{r}=[Q^{T} \bar{r} -z]^{+}. \]
\end{remark}

In that definition, we have to find "the fixed point" $\bar{r}$.
Typically, we use approximation sequence of $\bar{r}$ and $\bar{z}$:
\begin{equation}
\label{eq3.7} \left\{
\begin{array}{rcl}
\bar{r}_{0}&=&0,\\
\bar{z}_{0}&=&z,\\
\bar{r}_{n+1}&=&[Q^{T}\bar{r}_{n}-z]^{+},\;\;\;n\in{\mathbb N}\cup\{0\},\\
\bar{z}_{n+1}&=&z+(I-Q^{T})\bar{r}_{n+1}\;\;\;n\in{\mathbb
N}\cup\{0\}.
\end{array}
\right.
\end{equation}

It's easy to see that
\[ \lim_{n\longrightarrow+\infty}\bar{r}_n=\bar{r} \]
and
\[\lim_{n\longrightarrow+\infty}\bar{z}_{n}=\Pi_{Q}(z).\]

Using simple calculation, we can obtain equivalent formula for
$\bar{r}_{n+1}$:
\begin{remark}
\label{rem6}
\[\bar{r}_{n+1}=[Q^{T}\bar{r}_{n}-z]^{+}=[-(z+(I-Q^{T})\bar{r}_{n})+\bar{r}_{n}]^{+}=[\bar{z}_{n}+\bar{r}_{n}]^{+}\]
\end{remark}

Now we define another sequence starting from the same point:
\begin{equation}
\label{lab20}
\left\{\begin{array}{rcl} z_{0}&=&z,\\
z_{n+1}&=&z_{n}+(I-Q^{T})[-z_{n}]^{+}\;\;\;n\in{\mathbb N}\cup\{0\}.\\
\end{array}\right.
\end{equation}

Once again simple calculation lead to obtain an equivalent formula:
\begin{remark}
\label{rem7}
\[z_{n+1}=z_{n}+(I-Q^{T})[-z_{n}]^{+}=z+(I-Q^{T})\sum_{i=0}^{n}[-z_{i}]^{+} \]
\end{remark}

Sequences $z_{n}$ and $\bar{z}_{n}$ look diffrent, but in fact thwy
are only diffrent representation of the same sequence.

\begin{lemma}
\label{lem1}
\begin{equation}
\forall_{z\in\mathbb{R}^{d}}\;\forall_{n\in{\mathbb
N}\cup\{0\}}\;\;\; z_{n}=\bar{z}_{n}
\end{equation}
\end{lemma}
\begin{proof}
The proof will be done using induction. For $n=0$ we have:
\[z_{0}=z=\bar{z}_{0}.\]
Now assume that $z_{i}=\bar{z}_{i}$ for $i=0,\ldots,n$. Then from
(\ref{eq3.7}) and Remark \ref{rem6} we have:
\begin{equation}
\label{lab313} \bar{r}_{i}=\bar{r}_{i-1}+[-\bar{z}_{i-1}]^{+}
\end{equation}
for $i=0,\ldots,n$.

Now we check:
\begin{eqnarray*}
\bar{z}_{n+1}-z_{n+1}&=&z+(I-Q^{T})\bar{r}_{n+1}-z_{n}+(I-Q^{T})\;[-z_{n}]^{+}\\
&=&\bar{z}_{n}+(I-Q^{T})\;([-\bar{z}_{n}+\bar{r}_i]^{+}-\bar{r}_{n})-z_{n}+(I-Q^{T})\;[-z_{n}]^{+}\\
&=&(\bar{z}_{n}-z_{n})+(I-Q^{T})\;([-\bar{z}_{n}+\bar{r}_{n}]^{+}-\bar{r}_{n}-[-z_{n}]^{+})\\
&=&(I-Q^{T})\;([-\bar{z}_{n}+\bar{r}_{n}]^{+}-\bar{r}_{n}-[-\bar{z}_{n}]^{+})
\end{eqnarray*}
Let us define:
\begin{equation}
\label{lab311}
R_{n}^{j}=[-\bar{z}_{n}^{j}+\bar{r}_{n}^{j}]^{+}-\bar{r}_{n}^{j}-[-\bar{z}_{n}^{j}]^{+},\;\;\;
j=1,\ldots,d.
\end{equation}
It is easy to check that if $\bar{z}_{n}^{j}\leq 0$ then
$R_{n}^{j}=0$. To finish the proof we need to check whether:
$R_{n}^{j}=0$ when $\bar{z}_{n}^{j}>0$.

Without the loss of generality, we can assume that $j=1$. Then:
\begin{eqnarray*}
\bar{z}_{n}^{1}&=&z_{n}^{1}\\
&=&z_{n-1}^{1}+[-z_{n-1}^{1}]^{+}-q_{21}[-z_{n-1}^{2}]^{+}+\ldots
-q_{d1}[-z_{n-1}^{d}]^{+}\\
&\leq&z_{n-1}^{1}+[-z_{n-1}^{1}]^{+}\\
&=&\bar{z}_{n-1}^{1}+[-\bar{z}_{n-1}^{1}]^{+}
\end{eqnarray*}
So, if $\bar{z}_{n}^{1}>0$ then $\bar{z}_{n-1}^{1}>0$. In the same
way we can proof that $\bar{z}_{i}^{1}>0$ for $i=n-1, \ldots, 0$. If
$\bar{z}_{0}^{1}>0$ then $\bar{r}_{1}^{1}=0$ and by (\ref{lab313})
we have that $\bar{r}_{n}^{1}=0$. So $R_{n}^{1}=0$.
\end{proof}

\begin{corollary}
\label{cor26}
\[\lim_{n\longrightarrow+\infty}z_{n}=\Pi_{Q}(z).\]
\end{corollary}


\mbox{}\\
 Faculty of Mathematics and Computer Science,\\
 Nicholas Copernicus University\\
 ul. Chopina 12/18,\\
 87--100 Toru\'n,\\
 Poland
\end{document}